\documentclass{amsart}
\usepackage[margin=1.5cm]{geometry}
\usepackage{amsmath,amssymb,amsthm,hyperref}
\usepackage{longtable,}
\usepackage{cleveref}
\usepackage{tikz-cd}
\numberwithin{equation}{section}
\newcommand{\alt}[1]{\mathrm{Alt}(#1)}

\newtheorem{thm}{Theorem}[section]
\newtheorem{cor}[thm]{Corollary}

\newtheorem{lem}[thm]{Lemma}

\newtheorem{rem}[thm]{Remark}

\theoremstyle{definition}

\newcounter{claim}[thm]

\newcommand{\sym}[1]{\mathrm{Sym}(#1)}

\newcommand{\nor}[2]{\textbf{N}_{{#1}}({{#2}})}

\DeclareMathOperator{\bz}{bz}

\title[Base size on partitions]{On the base size of the symmetric and the alternating group acting on partitions}
\author{Joy Morris}
\address{Department of Mathematics and Computer Science\\
University of Lethbridge\\
Lethbridge, AB\\
T1K 3M4\\
Canada}
\thanks{The first author was supported by the Natural Science and Engineering Research Council of Canada (grant RGPIN-2017-04905).}
\email{joy.morris@uleth.ca}

\author{Pablo Spiga}
\address{Pablo Spiga\\
Dipartimento di Matematica e Applicazioni\\ University of Milano-Bicocca,\\\
Via Cozzi 55, 20125 \\Milano, Italy}
\email{pablo.spiga@unimib.it}

\begin{document}

\begin{abstract}
Given three positive integers $n,a,b$ with $n=ab$, we determine the base size of the symmetric group and of the alternating group of degree $n$ in their action on the set of partitions into $b$ parts having cardinality $a$.
\end{abstract}

\maketitle

\section{Introduction}
Given a permutation group $G$ on $\Omega$, a \textit{\textbf{base}} for $G$ is a subset of $\Omega$ whose pointwise stabilizer is the identity. The \textit{\textbf{base size}} $\mathrm{b}(G)$ of $G$ is the smallest cardinality of a base for $G$. Bases and in particular bases of small cardinality are of paramount importance in computational group theory, because they are at the heart of many algorithms for dealing with permutation groups.  Interesting combinatorial applications of bases are also discussed in~\cite{BaC}. In particular, bases with respect to the automorphism group of a combinatorial object were introduced in combinatorics as ``distinguishing sets" and considerable research was completed on this problem in ignorance of the work that had been done by group theorists.

Most interest in the base size of primitive groups originated from the classic results of Jordan~\cite{jordan}, bounding the cardinality of a primitive group via its base size. This interest was spurred in the '90s by the Cameron-Kantor conjecture~\cite{cameron}: there exists an absolute constant $b$ with $\mathrm{b}(G)\le b$, for every almost simple primitive group in a non-standard action. (We refer to~\cite{cameron} for the definition of \textit{\textbf{standard action}}.) This conjecture was settled in the positive in~\cite{B6}; however, the refinement of Cameron~\cite{cameron1}, asking whether one can take $b:=7$ has required considerable more effort. The detailed analysis of Burness~\cite{B1,B2,B3,B4} on fixed point ratios has resulted in an answer Cameron's question in~\cite{B5,B7,B8}. From these papers, one can infer that most almost simple primitive groups in non-standard actions have base size $2$.

At this point, there are three natural  problems: first, it is interesting to pin down exactly the base size of almost simple groups in all non-standard actions; second, determine the base size of arbitrary primitive groups; third, compute the base size of almost simple groups in standard actions. Much has been done on the first question in the work for solving Cameron's question. For the second problem, Fawcett~\cite{F1} has investigated the base size of primitive groups of diagonal type. In this paper, we are interested in the third problem and we are concerned with the action of the alternating group and of the symmetric group of degree $n=ab$ on the set of partitions having $b$ parts of cardinality $a$. This is one of the two standard actions of the almost simple groups with socle an alternating group that have not been determined. The other is the action on $k$-subsets, where the best results to date are the asymptotic results (that are precise if $n$ is sufficiently large relative to $k$) in~\cite{CGGM,halasi}.

The action on partitions has been considered a few times previously in the literature~\cite{BCN,B9,B10,james}, either for dealing directly with the problem of computing the base size or for studying related combinatorial invariants. In this paper, we finally explicitly determine the base size in all cases.

\begin{thm}\label{thrm:main}
Let $a$ and $b$ be positive integers with $a,b\ge 2$ and let $\bz(a,b)$ denote the base size of the symmetric group of degree $n=ab$ in its action on partitions into $b$ parts of cardinality $a$. Then
\begin{enumerate}
\item\label{eqthrm:main1} when $a=2$, $\bz(a,2)$ is undefined, $\bz(a,3)=4$ and $\bz(a,b)=3$ for every $b\ge 4$,
\item\label{eqthrm:main2} when $b=2$, $\bz(4,b)=5$  and $\bz(a,b)=\lceil\log_b(a+3)\rceil+1$ for every $a\notin\{2,4\}$, 
\item\label{eqthrm:main3} when $a\ge 3$ and $b\ge 3$, $\bz(a,b)=\lceil \log_b(a+2)\rceil+1$, unless $(a,b)\in \{(3,6),(3,7),(4,7),(7,3)\}$ or $b=a+2$,
\item\label{eqthrm:main4} $\bz(3,6)=\bz(3,7)=\bz(4,7)=3$, $\bz(7,3)=4$ and, for $a\ge 3$, $\bz(a,a+2)=3$.
\end{enumerate}
\end{thm}

\begin{thm}\label{thrm:main2}
Let $a$ and $b$ be positive integers with $a,b\ge 2$ and let $\bz'(a,b)$ denote the base size of the alternating group of degree $n=ab$ in its action on partitions into $b$ parts of cardinality $a$. Then either $\bz'(a,b)=\bz(a,b)$, or one of the following holds:
\begin{enumerate}
\item\label{eqthrm2:main1} $(a,b)\in \{(2,3),(3,6),(3,7),(4,7),(7,3)\}$,
\item\label{eqthrm2:main2} $b=a+2$ and $a\ge 5$,
\item\label{eqthrm2:main3} $b\ge3$, $a=b^k-1$ for some $k\ge 2$, $b<k+\lfloor\frac{k+1}{2}\rfloor+2$ and $(b,k)\ne (4,2)$,
\item\label{eqthrm2:main4}$b=2$ and $a\in \{b^k-1,b^k-2\}$ for some $k\ge 2$.
\end{enumerate}
In each of these cases, $\bz'(a,b)=\bz(a,b)-1.$
\end{thm}

\subsection{Notation}
Given a finite set $\Omega$, we denote by $\sym{\Omega}$ the symmetric group on $\Omega$, and by $\alt{\Omega}$ the alternating group on $\Omega$. Moreover, when $|\Omega|=n$ and when only the cardinality of $\Omega$ is  relevant for our arguments, we simply write $\sym{n}$ and $\alt{n}$.

A partition $\Sigma$ of $\Omega$ is said to be {\em \textbf{regular}} or {\em \textbf{uniform}} if all parts in $\Sigma$ have the same cardinality. We say that the partition $\Sigma$ is an {\mathversion{bold}$(a,b)$}{\em \textbf{-regular partition}} if $\Sigma$ consists of $b$ parts each having cardinality $a$. In particular, $n=|\Omega|=ab$. 
A partition $\Sigma$ of $\Omega$ is said to be {\em \textbf{trivial}} if $\Sigma$ equals the universal relation $\Sigma=\{\Omega\}$ or if $\Sigma$ equals the equality relation $\Sigma=\{\{\omega\}\mid \omega\in \Omega\}$.  

Let $G$ be a permutation group on $\Omega$. Given a subset $\Delta$ of $\Omega$ and a partition $\Sigma$ of $\Omega$, we let 
\begin{align*}
\nor G {\Delta}&:=\{g\in G\mid \Delta^g= \Delta\},\\
\nor G {\Sigma}&:=\{g\in G\mid \Gamma^g\in \Sigma,\forall \Gamma\in \Sigma\}
\end{align*}
denote the stabilizer in $G$ of the subset $\Delta$ and of the partition $\Sigma$ of $\Omega$.

Let $G$ be a permutation group on $\Omega$. For $\Lambda = \{\omega_1, \ldots, \omega_k \} \subseteq \Omega$, we write
 $G_{(\Lambda)}$ or $G_{\omega_1 ,\omega_2 ,\ldots,\omega_k}$ for the \textit{\textbf{pointwise
stabilizer}} of $\Lambda$ in $G$. If $G_{(\Lambda)} = \{1\}$, then we say that $\Lambda$ is a \textit{\textbf{base}}. The size of a smallest possible base is known as the \textit{\textbf{base size}} of $G$ and it is customary to denote it by $\mathrm{b}(G)$ or (more precisely) by $\mathrm{b}_\Omega(G)$.

Let $a$ and $b$ be integers with $a,b\ge1$. We denote by $$\bz(a,b)$$ the base size of the symmetric group 
$\sym{ab}$ in its action on the $(a,b)$-regular partitions of $\{1,\ldots,ab\}$, and by $$\bz'(a,b)$$ the base size of the alternating group $\alt{ab}$ in its action on the same collection. It is clear that the stabilizer of a point in this action of $\sym{ab}$ is isomorphic to the wreath product $\sym{a}\,\mathrm{wr}\,\sym{b}$ endowed of its imprimitive action on $\{1,\ldots,ab\}$, and for the action of $\alt{ab}$ is the subgroup of even permutations.

It is not hard to show that $\sym{ab}$ and $\alt{ab}$ act faithfully on the set of $(a,b)$-regular partitions, unless $b=1$ or $a=b=2$. 

In the rest of the paper, to avoid degeneracies, we suppose
$$a\ge 2\hbox{ and }b\ge 2$$
and we let $\Omega$ be a set of cardinality $ab$.
Moreover, we let $\Delta:=\mathbb{Z}/b\mathbb{Z}$. We often identify $\Delta$ with the set $\{0,\ldots,b-1\}$. 
Whenever we are performing arithmetic on the elements of $\Delta$ in this paper, it should be understood that the arithmetic is being taken modulo $b$.

\subsection{Structure of the paper} The structure of the paper is straightforward. In Section~\ref{sec:prelims}, we give some preliminary facts and most importantly we explain our approach for proving Theorems~\ref{thrm:main} and~\ref{thrm:main2}. In Section~\ref{sec:bounds}, we prove an upper bound for $\bz(a,b)$ and $\bz'(a,b)$ and we prove an auxiliary lemma. In Section~\ref{sec:basege3}, we use the auxiliary result of the previous section for constructing various bases when $b\ge 3$; we then apply these constructions in Section~\ref{sec:basesizege3} to prove Theorems~\ref{thrm:main} and~\ref{thrm:main2} when $b\ge3$. The rest of the paper deals with the case $b=2$.

\section{Preliminaries}\label{sec:prelims}

Let $\Omega$ be a finite set of cardinality $ab$ and let $\Sigma_1,\ldots,\Sigma_\ell$ be $(a,b)$-regular partitions of the set $\Omega$.

Let $\omega$ be an element of   $\Omega$. Since $\Sigma_i$ is a partition of $\Omega$, for every $i\in \{1,\ldots,\ell\}$, there exists a unique $X_{i,\omega}\in \Sigma_i$ with $\omega\in X_{i,\omega}$. Therefore we have a natural function
\begin{center}
    \begin{tikzcd}[cells = {nodes={minimum width=3.5em, inner xsep=0pt}},
                   row sep=0pt]
\Omega \ar[r]           &   \Sigma_i,   \\
\omega      \ar[r, mapsto]   &  X_{i,\omega}.
\end{tikzcd}
\end{center}
By combining these mappings for each $i\in \{1,\ldots,\ell\}$, we obtain a natural function
\begin{center}
    \begin{tikzcd}[cells = {nodes={minimum width=3.5em, inner xsep=0pt}},
                   row sep=0pt]
\theta:\Omega \ar[r]           &   \Sigma_1\times \cdots\times \Sigma_\ell,   \\
\omega      \ar[r, mapsto]   & \omega^\theta:= (X_{1,\omega},\ldots,X_{\ell,\omega}).
\end{tikzcd}
\end{center}
Since $b=|\Sigma_i|$ for every $i\in \{1,\ldots,\ell\}$, we may identify each $\Sigma_i$ with the set $\Delta=\mathbb{Z}/b\mathbb{Z}=\{0,1,\ldots,b-1\}$. In this way, we obtain a function 
$$\Omega\to \Delta^\ell=\left(\mathbb{Z}/b\mathbb{Z}\right)^\ell.$$

Suppose now that $\Sigma_1,\ldots,\Sigma_\ell$ satisfy the hypothesis:
\begin{equation}\label{uglyyy}
\forall i\in \{1,\ldots,\ell\}, \forall X_i\in \Sigma_i,\quad |X_1\cap X_2\cap\cdots\cap X_\ell|\le 1.
\end{equation}
Now,~\eqref{uglyyy} guarantees that the mapping  $\Omega\to \Sigma_1\times \cdots \times \Sigma_\ell$ is injective. In particular, when~\eqref{uglyyy} holds, by identifying $\Sigma_1\times\cdots \times \Sigma_\ell$ with $\Delta^\ell$, we may identify $\Omega$ as a subset of $\Delta^\ell$.

When dealing with the base size of $\alt\Omega$, it will also be important to consider a second hypothesis: there exists $(\bar{X}_1,\bar{X}_2,\ldots,\bar{X}_\ell)\in \Sigma_1\times \Sigma_2\times\cdots\times \Sigma_\ell$ such that $\forall i\in \{1,\ldots,\ell\}$ and $\forall X_i\in \Sigma_i$,
\begin{eqnarray}\label{uglyyyY}
|X_1\cap X_2\cap\cdots\cap X_\ell|\le 1&\textrm{when }(X_1,X_2,\ldots,X_\ell)\ne (\bar{X}_1,\bar{X}_2,\ldots,\bar{X}_\ell),\\
|X_1\cap X_2\cap\cdots\cap X_\ell|= 2&\textrm{when }(X_1,X_2,\ldots,X_\ell)= (\bar{X}_1,\bar{X}_2,\ldots,\bar{X}_\ell).\nonumber
\end{eqnarray}
Here,~\eqref{uglyyyY} guarantees that, except for the two elements in $\bar{X}_1\cap\bar{X}_2\cap\cdots\cap\bar{X}_\ell$, the mapping  $\Omega\to \Sigma_1\times \cdots \times \Sigma_\ell$ is injective. 

\begin{rem}\label{rempablo}{\rm If $\Sigma_1,\ldots,\Sigma_\ell$ are $(a,b)$-regular partitions of $\Omega$ witnessing that $\ell=\bz(a,b)$, then~\eqref{uglyyy} holds true. Indeed, if $|X_1\cap X_2\cap\cdots\cap X_\ell|\ge 2$ for some $X_i\in \Sigma_i$, then the transposition of $\sym{\Omega}$ interchanging two elements of $X_1\cap X_2\cap\cdots\cap X_\ell$ lies in $\bigcap_{i=1}^\ell\nor {\sym{ab}}{\Sigma_i}$, contradicting the fact that the pointwise stabilizer of $\Sigma_1,\ldots,\Sigma_\ell$ in $\sym{ab}$ is the identity.}
\end{rem}

\begin{rem}\label{remalt}{\rm If $\Sigma_1,\ldots,\Sigma_\ell$ are $(a,b)$-regular partitions of $\Omega$ witnessing that $\ell=\bz'(a,b)$, then either~\eqref{uglyyy} or~\eqref{uglyyyY} holds true. Indeed, if $|X_1\cap X_2\cap\cdots\cap X_\ell|\ge 3$ for some $X_i\in \Sigma_i$, then the $3$-cycle of $\alt{\Omega}$ rotating three elements of $X_1\cap X_2\cap\cdots\cap X_\ell$ lies in $\bigcap_{i=1}^\ell\nor {\alt{ab}}{\Sigma_i}$, contradicting the fact that the pointwise stabilizer of $\Sigma_1,\ldots,\Sigma_\ell$ in $\alt{ab}$ is trivial. Likewise, if $|X_1\cap X_2\cap\cdots\cap X_\ell|= 2$ for two different collections of $X_i\in \Sigma_i$, then the transposition of $\alt{\Omega}$ interchanging the two elements of each lies in $\bigcap_{i=1}^\ell\nor {\alt{ab}}{\Sigma_i}$, contradicting the fact that the pointwise stabilizer of $\Sigma_1,\ldots,\Sigma_\ell$ in $\alt{ab}$ is trivial. }
\end{rem}

\begin{rem}\label{rempablo2}{\rm
Suppose that $\Sigma_1,\ldots,\Sigma_\ell$ are $(a,b)$-regular partitions of $\Omega$ such that 
\begin{equation}\label{uglyy}
\forall i\in \{1,\ldots,\ell\}, \forall X_i\in \Sigma_i,\quad |X_1\cap X_2\cap\cdots\cap X_\ell|= 1.
\end{equation}
Now, the natural embedding $\theta:\Omega\to \Delta^\ell$ of $\Omega$ in $\Delta^\ell$ discussed above is also surjective.
This shows that there exists a natural one to one correspondence between $\Omega$ and $\Sigma_1\times \Sigma_2\times\cdots\times \Sigma_\ell$, which endows $\Omega$ of the natural structure of a Cartesian power $\Delta^\ell$, where as usual $\Delta=\mathbb{Z}/b\mathbb{Z}$. This is the typical point of view taken by a permutation group theorist dealing with Cartesian decompositions (see for instance~\cite{Praeger}): a Cartesian decomposition $\Omega=\Delta^\ell$ is a set of $\ell$ regular partitions  $\Sigma_1,\ldots,\Sigma_\ell$ satisfying~\eqref{uglyy}.}
\end{rem}

\begin{rem}\label{rempablopablo}{\rm
 Let $\Sigma_1,\ldots,\Sigma_\ell$ be $(a,b)$-regular partitions of $\Omega$ and let $\theta:\Omega\to \Sigma_1\times\cdots\times \Sigma_\ell$ be as above. Set $$G:=\bigcap_{i=1}^\ell\nor {\sym{\Omega}}{\Sigma_i}.$$

As $G$ fixes $\Sigma_i$ setwise for each $i$, $G$ acts on the set $\Sigma_i$ and hence $G$ has a natural action on $\Sigma_1\times \cdots \times \Sigma_\ell$. Now, the action of $G$ on $\Omega$ is compatible with the action of $G$ on $\Sigma_1\times \cdots \times \Sigma_\ell$, that is, 
$$\forall g\in G,\,\forall \omega\in \Omega,\quad (\omega^g)^\theta=(\omega^\theta)^g.$$

In particular, when~\eqref{uglyyy} holds true and we identify $\Omega$ with a subset of $\Delta_1\times\cdots\times \Delta_\ell$, the group $G$ is simply the subgroup of the Cartesian product $\sym{\Delta}\times\cdots\times\sym{\Delta}$ fixing the subset $\Omega$ of $\Delta^\ell$ setwise.
\begin{center} {\em This is  the point of view that we take in this paper.}\end{center}

In fact, from this we see that a collection $\Sigma_1, \ldots, \Sigma_\ell$ of $(a,b)$-regular partitions of $\Omega$ is a base for $\mathrm{Sym}(\Omega)$ if and only if, by viewing $\Omega$ as a subset of $\Delta^\ell$, the subgroup of $\sym{\Delta}\times \cdots \times \sym{\Delta}$ fixing $\Omega$ setwise is the identity. A similar comment applies for $\mathrm{Alt}(\Omega)$, when~\eqref{uglyyy} or~\eqref{uglyyyY} holds.}
\end{rem}

%
%
%

In order to work out base sizes, we need to rely on previous work that addresses the case that $a\le b$. We phrase this result in a way tailored to our application.

\begin{thm}[{{\cite[Theorem~2]{B10}}}]\label{thm:a<b}
Let $a$ and $b$ be positive integers with $2\le a\le b$.  Then
\[
\bz(a,b)=
\begin{cases}
\textrm{undefined}&\textrm{when }(a,b)=(2,2),\\
4&\textrm{when }(a,b)=(2,3),\\
3&\textrm{when }a=2 \textrm{ and }b\ge 4,\\
3&\textrm{when }a\ge 3 \textrm{ and }b=a+2,\\
3&\textrm{when }(a,b)\in \{(3,6),(3,7),(4,7)\},\\
\lceil\log_b(a+2)\rceil+1&\textrm{otherwise}.
\end{cases}
\]
\end{thm}
In other words, from Theorem~\ref{thm:a<b}, we see that except when $a=2$, or $b=a+2$, or for three exceptional cases,  we have $\bz(a,b)=\lceil\log_b(a+2)\rceil+1$. 

\section{Bounds}\label{sec:bounds}

Using the observations made in Section~\ref{sec:prelims}, we obtain a lower bound for $\bz(a,b)$.

\begin{lem}\label{lower-bound}
We have $\bz(a,b) \ge \lceil \log_b(a+2)\rceil +1.$
\end{lem}

\begin{proof}
Let $\ell=\bz(a,b)$. We identify $\Omega$ with a subset of $\Delta^\ell$.

In order to ensure that there are at least $ab$ distinct elements in $\Delta^\ell$ to identify with the elements of $\Omega$, we require $b^\ell \ge ab$. This implies that $\ell \ge \lceil \log_b(ab) \rceil=\lceil \log_b(a)\rceil +1$. 

Since $a\ge 2$ and $b\ge 2$, we have $$\lceil\log_b(a)\rceil \le \lceil\log_b(a+2)\rceil \le \lceil\log_b(a)\rceil+1.$$ Furthermore, $\lceil\log_b(a)\rceil = \lceil\log_b(a+2)\rceil$, completing the proof, unless either $a=b^{k-1}$ or $a+1=b^{k-1}$ for some $k \ge 2$. 

When $a=b^{k-1}$, we have $ab=b^k$. If $\ell=k$, then we must have $\Omega=\Delta^\ell$. Clearly, any permutation $\sigma \in (\sym{\Delta})^\ell$ fixes $\Delta^\ell$ setwise, contradicting Remark~\ref{rempablopablo}. Thus a base must have at least $\log_b(ab)+1=\lceil \log_b(a+2)\rceil +1$ elements when $a=b^{k-1}$.

When $a+1=b^{k-1}$, we have $ab=b^k-b$. If $\ell=k$, then $\Omega$ is identified with a subset of $\Delta^\ell$ that includes all but $b$ of its elements. Because we are dealing with regular partitions, for each of the coordinates every element of $\Delta$ must appear in exactly one of these $b$ elements. By relabelling the parts of the partitions if necessary, we may assume without loss of generality that these $b$ elements are $$(0,\ldots, 0), (1,\ldots, 1),\ldots, (b-1,\ldots, b-1).$$ 
Let $\sigma$ be the permutation that adds one to each coordinate of any element of $\Delta^\ell$. It is clear that $\sigma$ fixes the $b$ omitted elements setwise, and therefore fixes $\Omega$ setwise. So again we contradict Remark~\ref{rempablopablo}. Thus a base must have at least $$\log_b((a+1)b)+1=\lceil \log_b(a+2)\rceil +1$$ elements when $a+1=b^{k-1}$.
\end{proof}

When we are considering the alternating group, we initially require a simpler lower bound.

\begin{lem}\label{lower-bound-alt}
We have $\bz'(a,b) \ge \lceil \log_b(a+1)\rceil+1$.
\end{lem} 

\begin{proof}
Let $\ell=\bz'(a,b)$. As usual, any collection of $\ell$ $(a,b)$-regular partitions forming a base for $\mathrm{Alt}(\Omega)$ give rise to a natural map from $\Omega$ to $\Delta^\ell$.

 In order to ensure that there at least $ab$ distinct elements in $\Delta^\ell$ (with up to one element repeated twice, as explained in Remark~\ref{remalt}) to identify with the elements of $\Omega$, we require $b^\ell+1 \ge ab$.  Since $a\ge 2$ and $b\ge2$,  this implies that $$\ell \ge \lceil \log_b(ab-1) \rceil=\lceil \log_b(ab)\rceil=\lceil \log_b(a)\rceil +1.$$

Notice that $\lceil \log_b(a+1)\rceil=\lceil \log_b(a)\rceil$ (completing the proof) unless $a=b^k$ for some $k\ge1$. It remains only to show that, when $a=b^k$, we have $\ell >k+1=\lceil \log_b(a)\rceil+1$.

Suppose to the contrary that $\ell=k+1$. This would imply that we can identify $\Omega$ with $b^{k+1}$ elements of $\Delta^{k+1}$, where at least $b^{k+1}-1$ of these elements are distinct. In order for the elements of $\Delta^{k+1}$ to correspond to regular partitions as discussed in Section~\ref{sec:prelims}, however, the final element is certainly determined by the other $ab-1$ elements. Taking $ab-1=b^{k+1}-1$ distinct elements from $\Delta^{k+1}$ and requiring that the end result is a regular partition forces the final element to be the only remaining distinct element from $\Delta^{k+1}$. Now any non-identity even permutation in $(\sym{\Delta})^{k+1}$ fixes this collection setwise (and we have at least one such element because $k\ge 1$). Therefore we do not in fact have a base. This contradiction  shows that $\ell>k+1$, as claimed.
\end{proof}

For the next lemma and a number of the constructions that follow, we require a specific subset of $\Delta^{\ell+2}$ having $b$ elements, which we now define:
\begin{equation}
\label{eq:T}
T:=\{(\underbrace{0,\ldots, 0}_{\ell+1\textrm{ times}},1)\} \cup \{(x,\underbrace{x-1, \ldots, x-1}_{\ell \textrm{ times}}, x+1): 1 \le x \le b-1\}\subseteq \Delta^{\ell+2}.
\end{equation}

\begin{lem}\label{main-lemma}
Let $b$ and $\ell$ be positive integers with $b\ge 3$ and $\ell \ge 1$. Suppose there exists a subset $N$ of $\Delta^{\ell+2}$ with $|N|=ab$ and the following properties hold:
\begin{enumerate}
\item for every $i \in \Delta$ and for every $1 \le j \le \ell+2$, the number of elements of $N$ with $i$ in the $j$th coordinate is $a$; \label{condition:partition}
\item $T \subseteq N$;\label{condition:T}
\item if $(n_1, \ldots, n_{\ell+2}) \in N\setminus T$, then $n_{\ell+2} \neq n_1+1$; \label{condition:Tunique}
\item there is some constant $c\ne 1$ such that, for any $x \in \Delta$, the number of elements of $N$ whose first and last coordinates are equal to $x$ is $c$; and \label{condition:constant-equality}
\item for any $x\in \Delta$, the number of elements $(x,n_2, \ldots, n_{\ell+2})$ of $N$ with the property that $n_{\ell+2}=x+i$ is neither $1$ nor $c$, for any $i \in \Delta\setminus\{0,1\}$. \label{condition:no1s}
\end{enumerate}
Then $\bz(a,b) \le \ell+2$.
\end{lem}

\begin{proof}
As usual, we let $\Omega$ be a set of cardinality $ab$ and we identify the elements of $\Omega$ with the elements of $N$, with the understanding that, if $\omega \in \Omega$ is identified with $(n_1, \ldots, n_{\ell+2})$ in $N$, then it lies in part $n_i$ of the $i$th partition. By~\eqref{condition:partition}, the identification of $\Omega$ with $N\subseteq \Delta^{\ell+2}$ gives rise to $\ell+2$ $(a,b)$-regular partitions of $\Omega$.

Let $\sigma=(\sigma_1, \ldots, \sigma_{\ell+2}) \in (\sym{\Delta})^{\ell+2}$. By Remark~\ref{rempablopablo}, we need to show that, if $\sigma$ fixes $N$ setwise, then $\sigma$ is the identity; this will complete the proof. For convenience, we will let $N_i$ denote the set of elements in $N$ whose first coordinate is $i$.

Let $m\in \Delta$ be such that $0^{\sigma_1}=m$, so $N_0^\sigma=N_m$. Suppose for the time being that $m \neq 0$. By \eqref{condition:T} and \eqref{condition:Tunique}, $(0,\ldots, 0, 1)$ is the only element of $N_0$ with $1$ as its final coordinate. By \eqref{condition:constant-equality} and \eqref{condition:no1s}, for every other value that appears as the final coordinate of some element of $N_0$, the number of elements of $N_0$ having this value as their final coordinate is neither $1$ nor $c$. Similarly, by \eqref{condition:T} and \eqref{condition:Tunique}, $(m,m-1,\ldots, m-1, m+1)$ is the only element of $N_m$ with $m+1$ as its final coordinate. By \eqref{condition:constant-equality} and \eqref{condition:no1s}, for every other value that appears as the final coordinate of some element of $N_m$, the number of elements of $N_m$ having this value as their final coordinate is neither $1$ nor $c$. This forces $$(0,\ldots, 0, 1)^\sigma=(m,m-1,\ldots,m-1,m+1).$$ In particular, $1^{\sigma_{\ell+2}}=m+1$ and $0^{\sigma_2}=m-1$.

In a similar vein, by \eqref{condition:constant-equality} and \eqref{condition:no1s},  $N_{m+1}$ has exactly $c$ elements that end with $m+1$, and $N_1$ is the only set $N_i$ having exactly $c$ elements that end with $1$. Now, the fact that $1^{\sigma_{\ell+2}}=m+1$ forces $N_1^\sigma=N_{m+1}$ and therefore $1^{\sigma_1}=m+1$. In particular, following the same logic as in the previous paragraph and using $b\ge 3$, we obtain $$(1,0,\ldots, 0,2)^\sigma=(m+1,m,\ldots, m,m+2).$$ However, this implies that $0^{\sigma_2}=m$, contradicting the conclusion of the previous paragraph. 

The possibility remains that $m=0$, so that $(0,\ldots, 0,1)$ is fixed by $\sigma$. In this case, $1^{\sigma_{\ell+2}}=1$ and $0^{\sigma_i}=0$ for every $1 \le i \le \ell+1$. 

The argument so far forms a base case for induction. Suppose that $j^{\sigma_{\ell+2}}=j$. 
Then, since $N_j$ has $c$ elements whose final digit is $j$ (by \eqref{condition:constant-equality}) and no other set $N_i$ with $i \neq j$ has this property (by \eqref{condition:no1s}), we have $N_j^\sigma=N_j$. Thus, $j^{\sigma_1}=j$, and therefore by \eqref{condition:T} and \eqref{condition:Tunique} we must have $$(j,j-1,\ldots ,j-1,j+1)^\sigma=(j,j-1,\ldots ,j-1,j+1).$$ Thus $(j+1)^{\sigma_{\ell+2}}=j+1$. This shows inductively that $\sigma_1$ and $\sigma_{\ell+2}$ are the identity. In turn, this implies that each of the elements of $T$ is individually fixed by $\sigma$. Thus $(x,x-1, \ldots, x-1, x+1)$ is fixed by $\sigma$ for every $1 \le x \le b-1$. In particular, $(x-1)^{\sigma_i}=x-1$ for every $1 \le x \le b-1$ and every $2 \le i \le \ell+1$. This forces every $\sigma_i$ ($1\le i \le \ell+2$) to be the identity, meaning that $\sigma$ is the identity.
\end{proof}

\section{Base constructions when $b\ge 3$}\label{sec:basege3}

In this section, we construct a number of sets that fulfill the conditions of Lemma~\ref{main-lemma} under various conditions, thus producing upper bounds for $\bz(a,b)$. 
\begin{cor}\label{cor:maincase}
Let $b$ and $\ell$ be integers with $b\ge 3$ and $\ell \ge 1$. If $1 \le k \le b-2$ and either $r=0$ or $2 \le r \le b^\ell-2$, or if $k=b-1$ and $r=0$, then $\bz(kb^\ell+r,b) \le \ell+2$.
\end{cor}

\begin{proof}
We will find a subset $N$ of cardinality $ab$ of $\Delta^{\ell+2}$ and we show that $N$ satisfies all of the conditions of Lemma~\ref{main-lemma}, which completes the proof. 

Let $$T'=\{(\underbrace{0,\ldots, 0}_{\ell+1\textrm{ times}})\} \cup \{(x,\underbrace{x-1, \ldots, x-1}_{\ell \textrm{ times}}): 1 \le x \le b-1\} \subseteq \Delta^{\ell+1}.$$ Let $X$ be any subset of cardinality $r$ of $\{0\}\times\Delta^{\ell}$  containing neither $(0,\ldots, 0)$ nor $(0, b-1,\ldots, b-1)$. Such a subset exists because $|\{0\}\times \Delta^\ell|=b^\ell\ge r+2$. Let $\rho \in (\sym{\Delta})^{\ell+2}$ be the permutation that adds $1$ to every coordinate.

We let $M_1=T$ as defined in~\eqref{eq:T}, and $$M_0=\{(n_1, \ldots, n_{\ell+1},n_1): (n_1, \ldots, n_{\ell+1}) \in \Delta^{\ell+1}\setminus T'\}\subseteq \Delta^{\ell+2}.$$ Also, for $2 \le t \le k$, take
$$M_t=\{(n_1, \ldots, n_{\ell+1},n_1+t): (n_1, \ldots, n_{\ell+1})\in \Delta^{\ell+1}\}\subseteq \Delta^{\ell+2}$$ and let
$$M_{k+1}=\{(n_1, \ldots, n_{\ell+1},n_1+k+1): (n_1, \ldots, n_{\ell+1})\in X^{\rho^{n_1}}, 0 \le n_1 \le b-1\}\subseteq \Delta^{\ell+2}.$$ Notice that, regardless of $t$, 
\begin{center}
$(\dag):\qquad$ for any element in $M_t$, subtracting its first coordinate from its final coordinate yields $t$.
\end{center} We use this important property in what follows.

Now let $$N=\bigcup_{t=0}^{k+1}M_t.$$   
Note that $|M_1|=b$, and $|M_0|=b^{\ell+1}-b$, while $|M_t|=b^{\ell+1}$ for $2 \le t \le k$, and $|M_{k+1}|=br$. Using~$(\dag)$ and noting $k+1<b$ unless $r=0$, it is easy to see  that these sets are disjoint, so $$|N|=b+(b^{\ell+1}-b)+(k-1)b^{\ell+1}+br=kb^{\ell+1}+br=ab.$$ 
We verify the conditions of Lemma~\ref{main-lemma}.

To verify \eqref{condition:partition} of Lemma~\ref{main-lemma}, observe that in each of $M_0 \cup M_1$ and $M_t$ with $2 \le t \le k$, there are $b^\ell$ elements that have $i$ in the $j$th coordinate, while $M_{k+1}$ has $r$ elements that have $i$ in the $j$th coordinate. 
Part~\eqref{condition:T} of Lemma~\ref{main-lemma} is obvious from our definition of $M_1=T$, and part~\eqref{condition:Tunique} follows from~$(\dag)$. Part~\eqref{condition:constant-equality} follows from our construction of $M_0$, with $c=b^\ell-1$. 

We now verify~\eqref{condition:no1s} of Lemma~\ref{main-lemma}. For any $x \in \Delta$, the number of elements of $N$ satisfying the property defined in~\eqref{condition:no1s} is the number of elements of $M_i$ whose first coordinate is $x$. We have just observed that this is $b^\ell-1$ when $i=0$, $1$ when $i=1$, $b^\ell$ when $2 \le i \le k$, and $r$ when $i=k+1$. Since $1$, $b^\ell$, $r$, and $b^{\ell}-1$ are all distinct, the proof is complete.
\end{proof}

Next we deal with the situation where $a=kb^\ell+1$ with $1 \le k \le b-2$. We will require the additional hypothesis that $b^\ell>4$ in this result.

\begin{cor}\label{cor:r=1}
Let $b$ and $\ell$ be integers with $b\ge 3$, $\ell\ge 1$ and $b^\ell>4$. If $1 \le k \le b-2$, then $\bz(kb^\ell+1,b) \le \ell+2$.
\end{cor}

\begin{proof}
We will find a subset $N$ of $\Delta^{\ell+2}$ and show that $N$ satisfies all of the conditions of Lemma~\ref{main-lemma}, which completes the proof. 

Let $$T'=\{(\underbrace{0,\ldots, 0}_{\ell+1 \textrm{ times}})\} \cup \{(x,\underbrace{x-1, \ldots, x-1}_{\ell\textrm{ times}}): 1 \le x \le b-1\} \subseteq \Delta^{\ell+1}.$$ Let $X=\{x_1,x_2\}$ be any subset of cardinality $2$ of $\{0\}\times\Delta^{\ell}$  containing neither $(0,\ldots, 0)$ nor $(0, b-1,\ldots, b-1)$. Such a subset exists because $|\{0\}\times \Delta^\ell|=b^\ell>3$. Let $\rho \in (\sym{\Delta})^{\ell+2}$  and $\rho'\in (\sym{\Omega})^{\ell+1}$ be the permutations that add $1$ to every coordinate.

We let $M_1=T$ as defined in~\eqref{eq:T}, and 
\begin{align*}
M_0=\{(n_1, \ldots, n_{\ell+1},n_1):& (n_1, \ldots, n_{\ell+1}) \in\Delta^{\ell+1}\setminus T',\,(n_1,\ldots,n_{\ell+1})\ne x_1^{(\rho')^{n_1}},\,0\le n_1\le b-1\}.
\end{align*} Also, for $2 \le t \le k$, take
$$M_t=\{(n_1, \ldots, n_{\ell+1},n_1+t): (n_1, \ldots, n_{\ell+1})\in \Delta^{\ell+1}\}$$ and, for $j\in \{1,2\}$, let
$$M_{k+1,j}=\{(n_1, \ldots, n_{\ell+1},n_1+k+1): (n_1, \ldots, n_{\ell+1})=x_j^{\rho^{n_1}}, 0 \le n_1 \le b-1\}.$$ Notice that, regardless of $t$ and $j$, for any element in $M_t$ or $M_{t,j}$, subtracting its first coordinate from its final coordinate yields $t$.

Now let $$N=\bigcup_{t=0}^{k}M_t\cup M_{k+1,1}\cup M_{k+1,2}.$$ 
Note that $|M_1|=b$, and $|M_0|=b^{\ell+1}-2b$, while $|M_t|=b^{\ell+1}$ for $2 \le i \le k$, and $|M_{k+1,1}|=|M_{k+1,2}|=b$. It is easy to see (by our observation about the elements of $M_t$ and $M_{k+1,j}$, and noting $k+1<b$) that these sets are disjoint, so $$|N|=b+(b^{\ell+1}-2b)+(k-1)b^{\ell+1}+2b=kb^{\ell+1}+b=ab.$$ 

We verify the conditions of Lemma~\ref{main-lemma}.  
To verify \eqref{condition:partition} of Lemma~\ref{main-lemma}, observe that in each of $M_0 \cup M_1\cup M_{k+1,1}$ and $M_t$ with $2 \le t \le k$, there are $b^\ell$ elements that have $i$ in the $j$th coordinate, while $M_{k+1,2}$ has $1$ element that has $i$ in the $j$th coordinate. In fact, $b^\ell+(k-1)b^\ell+1=a$ and hence~\eqref{condition:partition} is verified.

%
%

Part~\eqref{condition:T} of Lemma~\ref{main-lemma} is obvious from our definition of $M_1=T$, and part~\eqref{condition:Tunique} follows from our observation about the sets $M_t$ and $M_{k+1,j}$. Part~\eqref{condition:constant-equality} follows from our construction of $M_0$, with $c=b^\ell-2$. As $b^\ell>4$, we have $c\ne 1$.

We now verify~\eqref{condition:no1s} of Lemma~\ref{main-lemma}.  For any $x \in \Delta$, the number of elements of $N$ satisfying the property defined in~\eqref{condition:no1s} is the number of elements of $M_t$ (or in the union of $M_{k+1,1}$ and $M_{k+1,2}$) whose first coordinate is $x$. We have just observed that this is $b^\ell-2$ when $i=0$, $b^\ell$ when $2 \le i \le k$, $2$ when $i=k+1$, and it is $1$ when $i=1$. Since $b^\ell>4$, these values are all distinct, so this condition is satisfied.
\end{proof}

In the next result  we deal with the possibility that $3 \le a \le b^\ell$. 

\begin{cor}\label{cor:small}
Let $b$ and $\ell$ integers with $b\ge 3$ and $\ell \ge 1$. If  $3 \le a\le b^\ell$, then $\bz(a,b) \le \ell+2$.
\end{cor}

\begin{proof}
We will find a subset $N$ of $\Delta^{\ell+2}$ and show that $N$ satisfies all of the conditions of Lemma~\ref{main-lemma}. 

Let $$V=\{(0,
\underbrace{b-1,\ldots, b-1}_{\ell \textrm{ times}})\} \cup \{(\underbrace{x,\ldots, x}_{\ell+1 \textrm{ times}}): 1 \le x \le b-1\} \subseteq \Delta^{\ell+1}.$$
Let $X$ be any subset of cardinality $a-2$ of $\{0\}\times\Delta^{\ell}$  containing neither $(0,\ldots, 0)$ nor $(0, b-1,\ldots, b-1)$. Such a subset exists because $|\{0\}\times \Delta^\ell|=b^\ell\ge a$. Let $\rho \in (\sym{\Delta})^{\ell+2}$ be the permutation that adds $1$ to every coordinate.

We let $M_1=T$ as defined in~\eqref{eq:T}, and $$M_0=\{(n_1, \ldots, n_{\ell+1},n_1): (n_1, \ldots, n_{\ell+1}) \in V \cup X^{\rho^{n_1}},\,0\le n_1\le b-1\}.$$  Notice that, for any element in $M_t$, subtracting its first coordinate from its final coordinate yields $t$.

Now let $N=M_0\cup M_1$. 
Note that $|M_1|=b$, and $|M_0|=b+b(a-2)=b(a-1)$. It is easy to see (by our observation about the elements of $M_t$) that these sets are disjoint, so $|N|=ab.$

We verify the conditions of Lemma~\ref{main-lemma}.  
To verify \eqref{condition:partition} of Lemma~\ref{main-lemma}, notice first that in $T\subseteq \Delta^{\ell+2}$ and in $V\subseteq \Delta^{\ell+1}$, every element of $\Delta$ appears exactly twice in each of the first $\ell+1$ coordinates (which are all of the coordinates for $V$). 
Also, in the final coordinate each element of $\Delta$ appears exactly once in $T$, and exactly once in the portion of $M_0$ that is derived from $V$. In the portion of $M_0$ that is derived from $X$, since the action of $\rho$ takes each coordinate through every element of $\Delta$ for each initial element in $X$, every element appears $a-2$ times. Thus every element of $\Delta$ appears $a$ times altogether in each coordinate.
This completes the verification of \eqref{condition:partition} of Lemma~\ref{main-lemma}.

Part~\eqref{condition:T} of Lemma~\ref{main-lemma} is obvious from our definition of $M_1=T$, and part~\eqref{condition:Tunique} follows from our observation about the first and last coordinates in the elements of $M_0$. Part~\eqref{condition:constant-equality} follows from our construction of $M_0$, with $c=1+a-2=a-1\ge 2$.

We now verify~\eqref{condition:no1s} of Lemma~\ref{main-lemma}. Let $i\in \Delta\setminus\{0,1\}$.  For any $x \in \Delta$, the number of elements $(x,n_2,\ldots,n_{\ell+2})\in N$  with $n_{\ell+2}=x+i$ is zero because of the definition of $M_0$ and $M_1$. As $c=a-1\ne 0$, condition~\eqref{condition:no1s} is satisfied.
\end{proof}

\section{Base sizes when $b \ge 3$}\label{sec:basesizege3}

We begin with a result that allows us to determine some base sizes if we know others.

\begin{lem}\label{lem:complement}
Let  $a'$ and  $b$ be integers with $a',b\ge 2$. If there exists an integer $\kappa$ with
$$\bz(a',b)\le \kappa\le \lceil \log_b(b^{\kappa}-a'b)\rceil,$$
then $\bz(b^{\kappa-1}-a',b)=\kappa$.
\end{lem}

\begin{proof}
Notice  that $\kappa=\log_b(b^\kappa) \ge \lceil \log_b(b^{\kappa}-a'b)\rceil$ and hence the second inequality in the hypothesis of this lemma yields $$\lceil \log_b(b^{\kappa}-a'b)\rceil= \kappa.$$

Let $\Omega_2$ be a set of cardinality $b(b^{\kappa-1}-a')$. Given a base of cardinality $\bz(b^{\kappa-1}-a',b)$ for the action of $\sym{\Omega_2}$ on $(b^{\kappa-1}-a',b)$-regular partitions, the set $\Omega_2$ can be identified with a certain subset of $\Delta^{\bz(b^{\kappa-1}-a',b)}$. However, in order to have sufficient elements in $\Delta^{\bz(b^{\kappa-1}-a',b)}$ to be in one-to-one correspondence with the elements of $\Omega_2$, we must have $$b^{\bz(b^{\kappa-1}-a',b)} \ge b^{\kappa}-a'b.$$ Equivalently, $\bz(b^{\kappa-1}-a',b) \ge \kappa$ by the observation in the previous paragraph.

To complete the proof, we construct a base of cardinality $\kappa$ for the action of $\sym{\Omega_2}$ on the set of $(b^{\kappa-1}-a',b)$-regular partitions. 

Let $\Omega_1$ be a set of cardinality $a'b$. Now, given a base of cardinality $\bz(a',b)$ for the action of $\sym{\Omega_1}$ on $(a',b)$-regular partitions, we obtain an embedding of $\Omega_1$ in $\Delta^{\bz(a',b)}$. Let $N_1\subseteq \Delta^{\bz(a',b)}$ be the image of this embedding and let $$N_1':=\{(n_1,\ldots, n_{\bz(a',b)},\underbrace{n_1,\ldots, n_1}_{\kappa-\bz(a',b) \textrm{ times}}):(n_1,\ldots,n_{\bz(a',b)})\in N_1\}.$$
Since every element of $\Delta$ appears $a'$ times in each coordinate among elements of $N_1$, every element will appear $a'$ times in each coordinate among elements of $N_1'$. Now take $N_2:=\Delta^{\kappa}\setminus N_1'$. Again, it is clear that each element of $\Delta$ will appear $b^{\kappa-1}-a'$ times in each coordinate. Moreover, $|N_2|=b^{\kappa}-|N_1'|=b(b^{\kappa-1}-a')=|\Omega_2|$. In particular, from Remark~\ref{rempablopablo}, we need to show that, if a permutation $\sigma \in (\sym{\Delta})^{\kappa}$ fixes $N_2$ setwise, then $\sigma$ must be the identity. 

If some $\sigma \in (\sym{\Delta})^{\kappa}$ fixes $N_2$ setwise, then it also fixes $N_1'$ setwise. Therefore, its projection onto the first $\bz(a',b)$ coordinates is a permutation in $(\sym{\Delta})^{\bz(a',b)}$ that fixes $N_1$ setwise. However, from Remark~\ref{rempablopablo}, the setwise stabilizer of $N_1$ in $(\sym{\Delta})^{\bz(a',b)}$ is the identity  and hence $\sigma$ must act as the identity on the first $\bz(a',b)$ coordinates. By our construction of $N_1'$, the action of $\sigma$ on each of the final $\kappa-\bz(a',b)$ coordinates must be identical to its action on the first coordinate, which is the identity. So $\sigma$ must be the identity.
\end{proof}

We pull our results together to prove Theorem~\ref{thrm:main} when $b \ge 3$.

\begin{proof}[Proof of 
Theorem~$\ref{thrm:main}$ parts~$\eqref{eqthrm:main1}$,~$\eqref{eqthrm:main3}$ and~$\eqref{eqthrm:main4}$.]
When $a=2$, the result follows directly from Theorem~\ref{thm:a<b}. For the rest of the proof, we suppose $a\ge 3$.

Let $\ell=\lfloor \log_b(a)\rfloor\ge 0$.
Then we can write $a$ as $kb^\ell+r$, for some $1 \le k \le b-1$ and some $0 \le r \le b^\ell-1$. 

If $\ell=0$ or if $\ell=1$ and $(k,r)=(1,0)$, then $a\le b$ and the result follows from Theorem~\ref{thm:a<b}. Henceforth we may assume that $\ell \ge 1$ and, when $\ell=1$, $(k,r)\ne (1,0)$, that is, $$a > b.$$

Observe that, from Lemma~\ref{lower-bound}, we have
$\bz(a,b)\ge \lceil\log_b(a+2) \rceil+1.$

If $b^\ell\le 4$, then using $\ell=\lfloor \log_b(a)\rfloor$ and $b \ge 3$, we deduce that either  $b=3$ and $4 \le a\le 8$, or $b=4$ and $5 \le a \le 15$. When $(a,b)\in \{(6,3),(6,4),(8,4),(10,4),(12,4)\}$, we deduce from Corollary~\ref{cor:maincase} that
$\bz(a,b)\le 3=\lceil\log_b(a+2)\rceil+1$ and hence the result follows from Lemma~\ref{lower-bound}. When $(a,b)\in \{(8,3),(15,4)\}$, as $a\le b^2$, we deduce from Corollary~\ref{cor:small} that
$\bz(a,b)\le 2+2= 4=\lceil\log_b(a+2)\rceil+1$ and hence the result follows from Lemma~\ref{lower-bound}. When $(a,b)\in \{(4,3),(5,3),(7,3),(5,4),(7,4),(9,4)\}$, we have verified with the computer algebra system magma~\cite{magma} that
\begin{align*}
\bz(4,3)&=3=\lceil\log_3(4+2)\rceil+1,\,\,
\bz(5,3)=3=\lceil\log_3(5+2)\rceil+1,\\
\bz(7,3)&=4=\lceil\log_3(7+2)\rceil+2\,\, \textrm{(which is one of the exceptions in Theorem~\ref{thrm:main}~\eqref{eqthrm:main3})},\\
\bz(5,4)&=3=\lceil\log_4(5+2)\rceil+1,\,\,
\bz(7,4)=3=\lceil\log_4(7+2)\rceil+1,\,\, \bz(9,4)=3=\lceil\log_4(9+2)\rceil+1.
\end{align*}
 It remains to consider $(a,b)\in \{(11,4),(13,4),(14,4)\}$. When $(a,b)=(11,4)$, we apply Lemma~\ref{lem:complement} with $a':=5$ and $\kappa:=3$ (because $\bz(5,4)=3$ and $\lceil \log_4(4^3-5\cdot 4)\rceil=3$) and we obtain $\bz(11,4)=3=\lceil\log_4(11+2)\rceil+1$. When $(a,b)=(13,4)$, we apply Lemma~\ref{lem:complement} with $a':=3$ and $\kappa:=3$ (because $\bz(3,4)=3$ from Theorem~\ref{thm:a<b} and $\lceil \log_4(4^3-3\cdot 4)\rceil=3$) and we obtain $\bz(13,4)=3=\lceil\log_4(13+2)\rceil+1$. When $(a,b)=(14,4)$, we apply Lemma~\ref{lem:complement} with $a':=2$ and $\kappa:=3$ (because $\bz(2,4)=3$ from Theorem~\ref{thm:a<b} and $\lceil \log_4(4^3-2\cdot 4)\rceil=3$) and we obtain $\bz(14,4)=3=\lceil\log_4(14+2)\rceil+1$.

 Henceforth we may assume that $b^\ell>4$ and, in particular, we may apply also Corollary~\ref{cor:r=1}. 
We now divide the proof depending on whether $k\ne b-1$, or $k=b-1$.

\smallskip

\noindent\textsc{Case $k\ne b-1$.} Since $b^\ell>4$, this implies $b^\ell<a+2<b^{\ell+1}-3$, and hence $\lceil \log_b(a+2)\rceil=\ell+1$. Thus our aim in this case will be to prove that $\bz(a,b)=\ell+2$.

If $ r \notin\{ 1, b^\ell-1\}$, then by Corollary~\ref{cor:maincase} we have $\bz(a,b)=\bz(kb^\ell+r,b) \le \ell+2$. Furthermore, by Lemma~\ref{lower-bound}, we have $\bz(a,b) \ge \ell+2$. Therefore, $\bz(a,b)=\ell+2$, as claimed. 

If $r=1$, then by Corollary~\ref{cor:r=1} we have $\bz(a,b)=\bz(kb^\ell+1,b) \le \ell+2$. Furthermore,  by Lemma~\ref{lower-bound}, we have $\bz(a,b) \ge \ell+2$. Therefore, $\bz(a,b)=\ell+2$, as claimed. 

When $r=b^\ell-1$, we apply Lemma~\ref{lem:complement}. Let $\kappa:=\ell+2$ and $a':=(b-k-1)b^\ell+1$. Notice that, by the previous paragraph, we have $\bz(a',b)=\bz((b-k-1)b^\ell+1,b)=\ell+2=\kappa$, because $1 \le b-k-1 \le b-2$. In particular, the first inequality in the hypothesis of Lemma~\ref{lem:complement} is satisfied (with equality). For the second inequality, $$b^{\ell+2}-a'b=b^{\ell+2}-((b-k-1)b^\ell+1)b=kb^{\ell+1}+b^{\ell+1}-b>b^{\ell+1},$$ since $k \ge 1$, so 
$\lceil\log_b(b^{\ell+2}-a'b)\rceil\ge \ell+2$, as required. Thus we may conclude that $\bz(b^{\ell+1}-((b-k-1)b^\ell+1),b)=\ell+2$. But $b^{\ell+1}-((b-k-1)b^\ell+1)=kb^\ell+(b^\ell-1)=a$. So we have established the desired result again.

\smallskip

\noindent\textsc{Case $k= b-1$.} If $r=0$, then Corollary~\ref{cor:maincase} again applies and, since $b \ge 3$, we have $\lceil \log_b(a+2)\rceil=\lceil\log_b(a)\rceil=\ell+1$. By Lemma~\ref{lower-bound}, we have $\bz(a,b) \ge \ell+2$. Therefore, $\bz(a,b)=\ell+2=\lceil \log_b(a+2)\rceil+1$, as claimed.

Now suppose that $r \ge 1$, and $a \le b^{\ell+1}-3$, so $b^{\ell+1}-b^{\ell}<a \le b^{\ell+1}-3$. 
Let $\kappa:=\ell+2$ and $a':=b^{\ell+1}-a$. We see that $3 \le a'<b^{\ell}$. Therefore, by Corollary~\ref{cor:small}, $\bz(a',b) \le \ell+2=\kappa$. Therefore, the first inequality in  Lemma~\ref{lem:complement} is satisfied. For the second inequality, notice that $b^{\ell+2}-a'b=b^{\ell+2}-(b^{\ell+1}-a)b=ab$, and $$\log_b(ab)>\log_b((b^{\ell+1}-b^\ell)b)=\log_b(b^{\ell+2}-b^{\ell+1})>\ell+1,$$ so $\lceil \log_b(b^{\ell+2}-a'b)\rceil \ge \ell+2=\kappa$, as required. Therefore by Lemma~\ref{lem:complement}, we see that $\bz(a,b)=\kappa=\ell+2$ again.

The remaining possibilities are $a=b^{\ell+1}-2$ and $a=b^{\ell+1}-1$. 
Suppose first $a=b^{\ell+1}-1$. In this case, $\lceil\log_b(a+2)\rceil=\ell+2$, so our aim is to prove that $\bz(a,b)=\ell+3$. That $\bz(a,b) \ge \ell+3$ follows from Lemma~\ref{lower-bound}.  We see that $3 \le a\le b^{\ell+1}$. 
Then Corollary~\ref{cor:small} (taking the $\ell$ of that result to be $\ell+1$) gives us $\bz(a,b) \le \ell+3$, as desired. 

Finally, suppose that $a=b^{\ell+1}-2$. In this case, $\lceil\log_b(a+2)\rceil=\ell+1$, so our aim is to prove that $\bz(a,b)=\ell+2$.  We apply Lemma~\ref{lem:complement} with $a'=2$ and $\kappa=\ell+2$. Notice that $b^{\ell+1}-2b>b^{\ell}$ since $\ell\ge 1$, so $\kappa=\ell+2 \le \lceil\log_b(b^{\ell+2}-2b)\rceil$ and the second inequality of Lemma~\ref{lem:complement} is satisfied.
By Theorem~\ref{thm:a<b}, we have $\bz(2,b)=4$ when $b=3$ and $\bz(2,b)=3$ when $b \ge 4$. Since $\ell \ge 1$, we have $\bz(2,b) \le \ell+2=\kappa$ and hence the first inequality in Lemma~\ref{lem:complement} is also satisfied.
\end{proof}

This result allows us to determine the corresponding result for $\bz'(a,b)$ in most cases. 

\begin{rem}\label{alt-easy}{\rm
We have $\bz'(a,b) \le \bz(a,b)$. This is because if $\Sigma$ is a collection of regular partitions that forms a base for the symmetric group, it must also form a base for the alternating group.}
\end{rem}

\begin{cor}\label{cor:cor}
Let $a$ and $b$ be integers with $a \ge 2$ and $b \ge 3$. Then one of the following holds
\begin{enumerate}
\item\label{cor:eq1} $\bz'(a,b)=\bz(a,b)$, or
\item\label{cor:eq2} $\bz'(a,b)=\bz(a,b)-1$ and $(a,b)\in \{(2,3),(3,6),(3,7),(4,7),(7,3)\}$, or $b=a+2$ and $a\ge 5$, or
\item\label{cor:eq4} $a=b^k-1$ for some $k \ge 2$.
\end{enumerate}
\end{cor}

\begin{proof}
Suppose first that $a\le b$. From~\cite[Remark~$2.8$]{B10}, we have $\bz(a,b)=\bz'(a,b)$, unless $(a,b)\in \{(2,3),(3,6),(3,7),(4,7)\}$, or $b=a+2$ and $a\ge5$. In all of the exceptional cases, again from~\cite[Remark~$2.8$]{B10}, we  have $\bz'(a,b)=\bz(a,b)-1$. In particular, when $a\le b$, we see that either part~\eqref{cor:eq1} or~\eqref{cor:eq2} holds true. For the rest of the proof we suppose $a> b$.

When $(a,b)=(7,3)$, we have verified with the help of a computer that 
$\bz'(a,b)=\bz(a,b)-1$. 

Suppose now that part~\eqref{cor:eq2} does not hold. Since $a>b$ and since we are excluding the cases listed in~\eqref{cor:eq2}, by Theorem~\ref{thrm:main}, we have $$\bz(a,b)=\lceil\log_b(a+2)\rceil+1.$$
Now,
by Remark~\ref{alt-easy}, we obtain $$\bz'(a,b) \le \lceil\log_b(a+2)\rceil+1.$$ By Lemma~\ref{lower-bound-alt}, we have $\bz'(a,b) \ge \lceil \log_b(a+1)\rceil+1$. Therefore, either $\bz'(a,b)=\bz(a,b)$ and part~\eqref{cor:eq1} holds, or $a=b^k-1$ for some $k\ge 2$ (because $a>b$) and part~\eqref{cor:eq4} holds.
\end{proof}

In the next two results, we deal completely with the possibility that $a=b^k-1$ assuming that $a,b \ge 3$.

\begin{lem}\label{lem:222}
Let $b$ and $k$ be integers with $b\ge 3$ and $k\ge 2$ and let $a=b^k-1$. If either of the following is true:
\begin{itemize}
\item $b<k+ \lfloor (k+1)/2\rfloor +2$ and $k \ge 3$, or 
\item $k=2$ and $b=3$,
\end{itemize}
then $\bz'(a,b)=\bz(a,b)-1$.
\end{lem}

\begin{proof}
By Theorem~\ref{thrm:main} for $b\ge 3$, we have $$\bz(a,b)=\lceil\log_b(a+2)\rceil+1=\lceil\log_b(b^k+1)\rceil+1=k+2.$$ By Lemma~\ref{lower-bound-alt} we have $$\bz'(a,b) \ge \lceil \log_b(a+1)\rceil+1=\lceil\log_b(b^{k})\rceil+1=k+1.$$ So, if we can find a base with $k+1$ elements, this will complete the proof. When $k=2$ and $b=3$, this is easily verified by computer. Indeed, the stabilizer in $\alt{24}$ of the following three $(8,3)$-regular partitions is the identity:
\begin{align*}
\Sigma_1:=&\{\{1,2,3,4,5,6,7,8\},\{9,10,11,12,13,14,15,16\},\{17,18,19,20,21,22,23,24\}\},\\
\Sigma_2:=&\{
\{ 1, 2, 8, 9, 12, 17, 18, 22 \},
    \{ 4, 5, 6, 10, 15, 16, 21, 23 \},
    \{ 3, 7, 11, 13, 14, 19, 20, 24 \}
\},\\
\Sigma_3:=&\{
    \{ 1, 5, 8, 13, 16, 17, 20, 21 \},
    \{ 2, 3, 4, 11, 12, 15, 18, 24 \},
    \{ 6, 7, 9, 10, 14, 19, 22, 23 \}\}.
\end{align*}

Assume now that $b<k+ \lfloor (k+1)/2\rfloor +2$ and $k \ge 3$ (so $k+1 \ge 4$). We aim to use Remarks~\ref{remalt} and~\ref{rempablopablo}. There are $b^{k+1}=ab+b$ distinct elements in $\Delta^{k+1}$. Our plan in this proof is to construct a certain set of cardinality $ab$; we do this by removing $b+1$ elements from $\Delta^{k+1}$ and by adding the duplicate of an element. Next, by considering the $k+1$ coordinates of $\Delta^{k+1}$, this set of cardinality $ab$ gives rise to $k+1$ $(a,b)$-regular partitions. Finally, using Remark~\ref{rempablopablo}, we prove that these $k+1$ regular partitions form a base for $\mathrm{Alt}(ab)$. Our construction depends on whether $k+1\le b-2$, or $k+1>b-2$. Set $$t:=\min\{k+1,b-2\}.$$

We add a second copy of $(0,\ldots, 0)$ in $\Delta^{k+1}$.  We remove the $t$ elements that have $i$ in every coordinate except the $i$th, and $0$ in the $i$th coordinate, where $1 \le i \le t$ (since $k+1 \ge t$ there are enough coordinates).

 Suppose first that $k+1 \le b-2$. So $t=k+1$ and the $t$ elements that we have removed so far are 
 \begin{align*}
 (0,\underbrace{1,\ldots,1}_{k\textrm{ times}}),\,
 (2,0,\underbrace{2,\ldots,2}_{k-1\textrm{ times}}),\,
 (3,3,0,\underbrace{3,\ldots,3}_{k-2\textrm{ times}}),\,\ldots,
 (\underbrace{k+1,\ldots,k+1}_{k\textrm{ times}},0).
 \end{align*} Since $b+1 \le k+1+\lfloor (k+1)/2\rfloor+1$ and we have removed $k+1$ elements, we have to remove at most $\lfloor (k+1)/2\rfloor +1$ more elements. We remove the elements that have $0$s in coordinates $2i-1$ and $2i$ and $k+1+i$ in every other coordinate, for $1 \le i \le b-k-2$. Since our bound on $b$ relative to $k$ implies that $2(b-k-2) \le k-1$, there are sufficient positions to do this. In practice, we are removing
 $$
 (0,0,\underbrace{k+2,\ldots,k+2}_{k-1 \textrm{ times}}),\,
 (k+3,k+3,0,0,\underbrace{k+3,\ldots,k+3}_{k-3 \textrm{ times}}),\ldots,
 (\underbrace{b-1,\ldots,b-1}_{2(b-k-2)-2 \textrm{ times}},0,0,b-1,b-1,\ldots,b-1).$$  Next we remove the element that has $k+1+i$ in coordinates $2i-1$ and $2i$ for $1 \le i \le b-k-2$, and $0$s in the remaining coordinates (again by our bound on $b$ relative to $k$ there will be at least two $0$s at the end of this string). Thus, we are removing the element
 $$
 (k+2,k+2,k+3,k+3,\ldots,b-1,b-1,0,\ldots,0).$$
  Finally we remove the element $$(1,2,3,\ldots,k+1)$$ that has $i$ in position $i$ for $1 \le i \le k+1$. Summing up, we have added one element and removed $k+1+(b-k-2)+1+1=b+1$, leaving us with $$ab+b+1-(b+1)=ab$$ elements. It is elementary to verify that in this set, for each coordinate $i$ and for each element $x$ of $\Delta$, $x$ appears in coordinate $i$ exactly $a$ times. Therefore, we obtain $k+1$ $(a,b)$-regular partitions. (Note that $k+1 \ge 4$ is also necessary in this construction.)

\smallskip 

Suppose now $b-2<k+1$ so that $t=b-2$. Then the $t$ elements we have removed in the first step of our construction are
\begin{align*}
 (0,\underbrace{1,\ldots,1}_{k\textrm{ times}}),\,
 (2,0,\underbrace{2,\ldots,2}_{k-1\textrm{ times}}),\,
 (3,3,0,\underbrace{3,\ldots,3}_{k-2\textrm{ times}}),\,\ldots,
 (\underbrace{b-2,\ldots,b-2}_{b-3\textrm{ times}},0,b-2,\ldots,b-2).
 \end{align*} 
Thus we need to remove $3$ more elements. We remove 
 \begin{itemize}
 \item the element $(b-1,b-1,0,\ldots,0)$ that has $b-1$ in each of the first $2$ coordinates, and $0$ in the remaining coordinates (since $k\ge 3$ there are at least two such coordinates); 
 \item the element $(0,0,b-1,\ldots,b-1)$ that has $b-1$ in each of the final $k-1$ coordinates and $0$ in the remaining $2$ coordinates; and 
 \item the element $(1,2,3,\ldots,b-2,0,\ldots,0)$ that has $i$ in coordinate $i$ for $1 \le i \le b-2$, and $0$ in the remaining $k-b+3 \ge 1$ coordinates. 
 \end{itemize}

Let $\sigma=(\sigma_1, \ldots, \sigma_{k+1}) \in (\sym{\Delta})^{k+1}$ and suppose that $\sigma$ fixes setwise the collection of elements of $\Delta^{k+1}$ that we have chosen. Since $(0,\ldots,0)$ is the only element that appears twice in our set, it must be fixed by $\sigma$, so $0^{\sigma_i}=0$ for $1 \le i \le k+1$. 

Suppose that $t=k+1$. Since the element that has $0$s in coordinates $2i-1$ and $2i$ and $k+1+i$ in every other coordinate, where $1 \le i \le b-k-2$, is the only element with $0$s in both of these coordinates that is not in our set, and $\sigma$ has to send it to an element that has $0$s in both of these coordinates, it must be fixed by $\sigma$. Likewise, the element that has $k+1+i$ in coordinates $2i-1$ and $2i$ for $1 \le i \le b-k-2$, and $0$s in the remaining coordinates is the only element with $0$ as its final two entries that is not in our set, so it must be fixed by $\sigma$. Now
for $1 \le i \le k+1$, since the element we removed in our first step is the only element with $i$th entry $0$ that is not in our set and has not already been fixed by $\sigma$, it too must be fixed by $\sigma$. Taking all of these together forces each $\sigma_i$ to fix at least all but one element of $\Delta$ (the element that appears in coordinate $i$ in the very last string we removed), which means that in fact $\sigma_i$ fixes every element of $\Delta$. So $\sigma$ is the identity.

Finally, suppose that $t=b-2<k+1$. For each $1 \le i\le b-2$, there are two elements not in our set that have $0$ in coordinate $i$; one of them has just one $0$ (in coordinate $i$) and $i$ everywhere else, while the other has at least two $0$s. Since $\sigma$ must send an element that has a single $0$ in coordinate $i$ to an element that has a single $0$ in coordinate $i$, each of the elements not in our set that has a single $0$ in coordinate $i$ must be fixed by $\sigma$. Similarly, the other elements not in our set that have $0$s in some coordinate $i$ where $1 \le i \le b-2$ must each be fixed by $\sigma$, and so the final element not in our set must be fixed by $\sigma$. Since every element of $\Delta$ appears at least once in each cordinate somewhere in these fixed elements, this forces $\sigma$ to be the identity, completing our proof.
\end{proof}

\begin{lem}\label{lem:22}
Let $b$ and $k$ be integers with $b\ge 3$ and $k\ge 2$ and let $a=b^k-1$. If either of the following is true:
\begin{itemize}
\item $b \ge k+\lfloor (k+1)/2\rfloor +2$, or 
\item $k=2$ and $b=4$,
\end{itemize}
then $\bz'(a,b)=\bz(a,b)$.
\end{lem}

\begin{proof}
By Theorem~\ref{thrm:main} for $b\ge3$, we have $\bz(a,b)=\lceil\log_b(a+2)\rceil+1=\lceil\log_b(b^k+1)\rceil+1=k+2$. By Lemma~\ref{lower-bound-alt} we have $$\bz'(a,b) \ge \lceil \log_b(a+1)\rceil+1=\lceil\log_b(b^{k+1})\rceil=k+1.$$ Also, by Remark~\ref{alt-easy} we have $\bz'(a,b) \le \bz(a,b)=k+2$. So if we can show that there is no base with $k+1$ elements, this will complete the proof. We aim to use Remarks~\ref{remalt} and~\ref{rempablopablo}. Given $k+1$ $(a,b)$-regular partitions of a set $\Omega$ of cardinality $ab$, we obtain a natural map from $\Omega$ to $\Delta^{k+1}$. Let $N$ be the image of this mapping. From Remark~\ref{alt-easy}, if the regular partitions form a base for $\mathrm{Alt}(ab)$, then $N$ has either cardinality $ab$ (and hence we may identify $\Omega$ with $N$) or $ab-1$ (and hence we may identify $\Omega$ with $N$ and with the duplicate of an element of $\Delta^{k+1}$). We consider a number of cases individually.

First suppose that $N$ does not include any duplicate elements from $\Delta^{k+1}$. Then $N$ uses all but $b$ elements from $\Delta^{k+1}$. By relabelling the parts in our partitions if necessary, we may assume without loss of generality that the elements not used by the base are $(i, \ldots, i)$ where $0 \le i \le b-1$ (otherwise we would not have regular partitions). Since $b \ge 3$, it is easy to see that the permutation $\sigma$ that acts as the permutation $(0\ 1\ 2)$ on each coordinate of $\Delta^{k+1}$ is an even permutation that fixes  setwise the set $\{(i,\ldots,i):i\in \Delta\}$, contradicting our assumption that we had chosen a base. Henceforth we may assume that $N$ must include a duplicate element; again without loss of generality (by relabelling the parts in our partitions) we may assume that this duplicated element is $(0, \ldots, 0)$. 

At this point, the action of $\sym{\Omega}$ on our base already includes the involution that exchanges the two elements of $\Omega$ that map to $(0,\ldots, 0)$, so if we indeed have a base for $\alt{\Omega}$ the action of $(\sym{\Delta})^{k+1}$ on the elements we choose must be trivial.
In order for our base to be comprised of $(a,b)$-regular partitions, the $b+1$ elements that we remove (after duplicating $(0,\ldots,0)$) must altogether include two $0$s and one of each other value $i$ ($1 \le i \le b-1$) in each coordinate. Also, $\sigma$ fixes the base setwise if and only if it fixed the collection of elements that are removed, setwise.

The first case we consider is $b+1 \ge 2(k+1)+2$, i.e. $b \ge 2k+3$. From the $b+1$ elements of $\Delta^{k+1}$ that we remove, at most $2(k+1)$ contain a $0$ since altogether there are two $0$s in each of the $k+1$ coordinates. If $b \ge 2k+3$ then there are at least two removed elements $i=(i_1,\ldots, i_{k+1})$ and $j=(j_1,\ldots, j_{k+1})$ such that $i_m, j_m\neq 0$ for $1 \le m \le k+1$. Neither $i_m$ nor $j_m$ appears in the $m$th coordinate of any other removed element. Therefore, taking $\sigma_m$ to be the element of $\sym{\Delta}$ that transposes $i_m$ and $j_m$, and $\sigma=(\sigma_1, \ldots, \sigma_{k+1})$, we see that not only does $\sigma$ exchange the two removed elements $i$ and $j$, but due to the lack of repetition of these coordinates in these coordinates amongst the other removed elements, $\sigma$ fixes every other element we have removed. Thus, we have a nontrivial permutation $\sigma \in (\sym{\Delta})^{k+1}$ that fixes any set we try to choose as a base, so no base with $k+1$ elements exists.

We can refine the argument of the previous paragraph slightly, to deal with the case where $b+1 \ge k+1+\lfloor (k+1)/2\rfloor+2$. This is because if any two of our removed elements that do contain at least one $0$ have their zeroes in exactly the same coordinates, then the element $\sigma$ of $(\sym{n})^{k+1}$ that transposes the $m$th coordinate of these two removed elements and fixes any other possible coordinate in the $m$th coordinate, will be a nontrivial permutation of the set we were trying to make a base. So in an actual base, our removed elements can include at most $k+1$ that have exactly one zero coordinate (in each possible coordinate), and another at most $\lfloor (k+1)/2\rfloor$ that have two or more zero coordinates. 

This completes the proof except in the small special case, which is easy to check by computer. 
\end{proof}

We conclude this section by proving Theorem~\ref{thrm:main2} when $b\ge3$.
\begin{proof}[Proof of 
Theorem~$\ref{thrm:main2}$ for $b\ge3$.]
From Corollary~\ref{cor:cor}, it suffices to consider the case $a=b^k-1$ with $k\ge 2$. Suppose that $\bz'(a,b)\ne \bz(a,b)$. Then, from Lemma~\ref{lem:22}, we must have $$b<k+\left\lfloor\frac{k+1}{2}\right\rfloor+2,$$
and $(b,k)\ne (4,2)$. When $k\ge 3$, Lemma~\ref{lem:222} yields $\bz'(a,b)=\bz(a,b)-1$ and the proof follows. Assume then $k<3$, that is, $k=2$. Therefore, $b<2+1+2=5$ and $b\in \{3,4\}$. As $(b,k)\ne (4,2)$, we deduce $b=3$ and Lemma~\ref{lem:222} gives $\bz'(a,b)=\bz(a,b)-1$.
\end{proof}

\section{Base sizes when $b=2$}

When $b=2$, since $\bz(2,2)$ does not exist we assume $a \ge 3$. The lower bound of Lemma~\ref{lower-bound} still applies, so $\bz(a,2) \ge \lceil\log_2(a+2)\rceil +1$. We begin with a lemma that deals with the cases where $a=2^i$ or $a=2^i-1$, for some $i \ge 3$.

\begin{lem}\label{lem:b=2}
If $a \in \{2^i,2^i-1: i \ge 3\}$, then $\bz(a,2)=i+2$.
\end{lem}

\begin{proof}
By Lemma~\ref{lower-bound}, we see that $\bz(a,2) \ge i+2$. Here $\Delta:=\{0,1\}$.

Define 
$$T_3=\{(0,0,0),(0,0,1),(0,1,0),(1,0,1)\}\subseteq \Delta^3.$$
Partition the elements of $\Delta^{i+1}$ into 
$$T_{i+1}:=\{(\underbrace{0,\ldots,0}_{i-2 \textrm{ times}},x,y,z)\in \Delta^{i+1}: (x,y,z) \in T_3\}\,\,\,\hbox{ and }\,\,\,T_{i+1}' = \Delta^{i+1}\setminus T_{i+1}.$$
Define $$T:=\{(n_1, \ldots, n_{i+2})\in \Delta^{i+2}: (n_1, \ldots, n_{i+1}) \in T_{i+1}\,\text{ and }\,n_{i+2}=1-n_{i+1}\},$$
and let 
$$N_1=T \cup \{(n_1, \ldots, n_{i+2})\in \Delta^{i+2}: (n_1, \ldots, n_{i+1}) \in T_{i+1}'\, \text{ and }\, n_{i+2}=n_{i+1}\}$$
and $$N_2=N_1\setminus \{(\underbrace{x,\ldots, x}_{i+2\textrm{ times}}): x \in\{0,1\}\}\subseteq \Delta^{i+2}.$$

By considering the $i+2$ coordinates of $\Delta^{i+2}$, the inclusion of $N_1$ (respectively, $N_2$) in $\Delta^{i+2}$ gives rise to a collection of $i+2$  partitions of a set of cardinality $|N_1|=2^{i+1}$ (respectively, $|N_2|=2^{i+1}-2$) having two parts. We claim that these partitions are in fact $(2^i,2)$-regular partitions (respectively, $(2^i-1,2)$-regular partitions).

Since $T_{i+1}$ and $T_{i+1}'$ partition $\Delta^{i+1}$, $0$ and $1$ each appear exactly $2^i$ times in each of the first $i+1$ coordinates among the elements of $N_1$. Since the final coordinate differs from the one before it in exactly $4$ elements of $N_1$ (two of which have a $0$ in this position and two of which have a $1$), $0$ and $1$ also each appear exactly $2^i$ times in the final coordinate. Similarly, $N_2$ corresponds to a collection of regular partitions, since $0$ and $1$ each occur once in each coordinate in the two elements that are removed from $N_1$ to yield $N_2$.

From Remark~\ref{rempablopablo}, to conclude the proof it remains to show that, if $\sigma=(\sigma_1, \ldots, \sigma_{i+2}) \in (\sym{\Delta})^{i+2}$ fixes either $N_1$ or $N_2$ setwise, we must have $\sigma=1$.  First observe that since $|T_{i+1}|=|T_3|=4$, we have $|T_{i+1}'|=2^{i+1}-4 \ge 12$, because $i \ge 3$. So $N_1$ has exactly $|T|=4$ elements in which the final coordinate and the one before it are not equal; those coordinates are equal in all of the elements of $N_1\setminus T$; and $|N_1\setminus T| \ge 12$. Similarly, $N_2$ has exactly $|T|=4$ elements in which the final coordinate and the one before it are not equal; those coordinates are equal in all of the elements of $N_2\setminus T$; and $|N_2\setminus T| \ge 10$. So in each case, the set $T$ of elements in which the final coordinate and the one before it are not equal, must be fixed setwise by $\sigma$. By definition of $T_3$, exactly one element of $T$ has a $1$ in coordinate $i$, so this element of $T$ must be fixed by $\sigma$. But this means that every $\sigma_j$ must be the identity, so $\sigma=1$, completing the proof.
\end{proof}

Putting this together with Lemma~\ref{lem:complement} and some calculations of small cases enables us to conclude the proof of Theorem~\ref{thrm:main} part~\eqref{eqthrm:main2}, that is, when $b=2$.

\begin{proof}[Proof of Theorem~$\ref{thrm:main}$ part~\eqref{eqthrm:main2}]Here $b= 2$.
For the sake of later arguments it will be important to note that $$\bz(3,2)=4=\lceil\log_2(3+3)\rceil+1$$ and $\bz(4,2)=5$ (one larger than the general result); a computer can easily verify these. In fact, even though $\bz(4,2)$ does not equal $\lceil\log_2(a+3)\rceil+1$, these can serve as the base cases for our induction. We now suppose $a\ge 5$.

Assume inductively that the result holds when $3 \le a < 2^i+r$ (with the exceptional value when $a=4$ as explained above), where $0 \le r \le 2^i-1$, so $i \ge 2$. Let $x=2^i+r$. We now divide the proof depending on the value of $r$.

\smallskip

\noindent\textsc{Case $r=0$ or $r=2^i-1$.} Here, the result holds for $\bz(x,2)$ using Lemma~\ref{lem:b=2}. 

\smallskip

\noindent\textsc{Case $1 \le r \le 2^i-3$.} Notice that $\lceil \log_2(x+3)\rceil=i+1$ in this case, so we are aiming to prove that $\bz(x,2)=i+2$. Take $\kappa=i+2$, and $a'=2^{i+1}-x$ in Lemma~\ref{lem:complement}. Then $3 \le a' \le 2^i-1$. By our inductive hypothesis together with the small cases mentioned previously, we know that $\bz(a',2)=\lceil \log_2(a'+3)\rceil +1$ unless $a'=4$ in which case $\bz(a',2)=5$. Take $\ell_1=\bz(a',2)$. Now, when $a' \le 2^i-1$ and $i \ge 2$, we have $\lceil\log_2(a'+3)\rceil \le i+1$, so $\ell_1 \le \kappa$. Also, when $a'=4$, since $x \ge 5$, we have $i \ge 3$ and hence $\kappa=i+2 \ge\ell_1=5$. Notice also that, when $a' \le 2^i-1$, we have $$\log_2(2^{i+2}-2a') \ge \log_2(2^{i+1})+1>i+1,$$ so $\lceil\log_2(2^\ell_2-2a')\rceil \ge \kappa$. Thus the conditions of Lemma~\ref{lem:complement} hold and we conclude that $\bz(a,2)=\kappa=i+2$, as desired. 

\smallskip

\noindent\textsc{Case $r=2^i-2$.} Here $x=2^{i+1}-2$ and $\lceil \log_2(x+3)\rceil=i+2$, so we are aiming to prove that $\bz(x,2)=i+3$. First we show that $\bz(x,2) \ge i+3$. By Lemma~\ref{lower-bound}, we know that $\bz(x,2) \ge i+2$. Since there are $2x=2^{i+2}-4$ elements of $\Omega$ in this case, if $\bz(x,2)=i+2$, then we would have to identify the elements of $\Omega$ with all but $4$ of the elements of $\Delta^{i+2}$. Thus, there would have to be a set of $4$ elements of $\Delta^{i+2}$ that is fixed setwise by no nontrivial $\sigma \in (\sym{\Delta})^{i+2}$. However,  applying Remark~\ref{rempablopablo} to this subset of $\Delta^{i+2}$ having $4$ elements, this would imply that $\bz(2,2) \le i+2$, whereas $\bz(2,2)$ is undefined. Thus, $\bz(x,2) \ge i+3$, as claimed.

Now we define a collection of $2^{i+2}-4$ elements of $\Delta^{i+3}$. Note that, since $i \ge 2$, we have $i+2 \ge 4$ and $i+3 \ge 5$. Define
$$T'=\{(\underbrace{0,\ldots,0}_{i+1\textrm{ times}},1),(1,\underbrace{0, \ldots, 0}_{i+1\textrm{ times}}),(0,\underbrace{1,\ldots, 1}_{i\textrm{ times}},0), (\underbrace{1, \ldots, 1}_{i+2\textrm{ times}})\} \subseteq \Delta^{i+2},$$
let $$T_1=\{(\underbrace{0,\ldots, 0}_{i+2\textrm{ times}}),(1,0,1,\underbrace{0, \ldots, 0}_{i-1\textrm{ times}})\} \subseteq \Delta^{i+2},$$ and let $T=\Delta^{i+2}\setminus (T'\cup T_1)$. 
Define
$$N=\{(n_1, \ldots, n_{i+2},n_1)\in \Delta^{i+3}: (n_1, \ldots, n_{i+2}) \in T\} \cup \{(n_1, \ldots, n_{i+2},1-n_1)\in \Delta^{i+3}: (n_1, \ldots, n_{i+2}) \in T_1\}.$$

By considering the $i + 3$ coordinates of $\Delta^{ i+3}$, the inclusion of $N$ in $\Delta^{ i+3}$ gives rise to a collection of $i + 3$ partitions of the set $N$ having two parts. 
It is straightforward to verify that these are in fact $(x,2)$-regular partitions. From Remark~\ref{rempablopablo}, we need to check that, if $\sigma=(\sigma_1, \ldots, \sigma_{i+3}) \in (\sym{\Delta})^{i+3}$ fixes $N$ setwise, then $\sigma$ is the identity. Since $$\{(n_1, \ldots, n_{i+2},1-n_1)\in\Delta^{i+3}: (n_1, \ldots, n_{i+2}) \in T_1\}$$ are the only two elements of $N$ whose first and last coordinates differ, they must be fixed setwise by $\sigma$. If $\sigma$ fixes either of them, then $\sigma$ must be the identity. The only other possibility is that $\sigma$ exchanges these two elements, which completely determines $\sigma$. But we see that for this $\sigma$, $(0,1,0,1,\ldots,1) \in N$ maps to $(1,\ldots, 1,0)$ which is not in $N$ since $(1,\ldots, 1) \in T'$. Thus $\sigma$ must be the identity, completing the proof.

\smallskip

This completes our inductive step, showing that the result is true for all $a \ge 5$.
\end{proof}

Now we consider $\bz'(a,2)$ when $a \ge 5$.

\begin{cor}\label{cor:b=2alt}
Suppose $a \ge 5$. Then $\bz'(a,2) =\bz(a,2)$ except possibly if $a=2^k-1$ or $a=2^k-2$ for some $k \ge 3$, in which case $\bz(a,2)-1 \le \bz'(a,2) \le \bz(a,2)$.
\end{cor}

\begin{proof}
By Theorem~\ref{thrm:main} and Remark~\ref{alt-easy} we have $$\bz'(a,2) \le \bz(a,2)=\lceil\log_2(a+3)\rceil+1.$$ By Lemma~\ref{lower-bound-alt}, we have $\bz'(a,2) \ge \lceil \log_2(a+1)\rceil+1$. Thus, $\bz'(a,2)=\bz(a,2)$ and we are done, except possibly if $\lceil\log_2(a+1)\rceil< \lceil\log_b(a+3)\rceil$.

If $\lceil\log_2(a+1)\rceil< \lceil\log_b(a+3)\rceil$, then $a=2^k-1$ or $2^k-2$, for some $k \ge 1$. Since $a \ge 5$, we have $k \ge 3$.
\end{proof}

We next consider $\bz'(2^k-1,2)$ and $\bz'(2^k-2,2)$.

\begin{lem}\label{lem:new}
Suppose $k \ge 3$ and let $a\in \{2^{k}-1,2^{k}-2\}$. Then $\bz'(a,2)=\bz(a,2)-1$.
\end{lem}

\begin{proof}
By Corollary~\ref{cor:b=2alt} and Theorem~\ref{thrm:main}, we know that $$\bz'(a,2) \ge k+1=\bz(a,2)-1,$$ when $a \in \{2^k-1,2^k-2\}$. We will construct a base of size $k+1$ to complete the proof.

When $a=2^k-1$, we will choose $ab=2^{k+1}-2$ elements of $\Delta^{k+1}$ so that exactly $2^{k+1}-3$ of them are distinct, and the setwise stabiliser of these elements is trivial. We take two copies of $(0,\ldots,0)$, and omit any three elements that leave us with regular partitions -- for example, 
$$(1,
\underbrace{0,\ldots, 0}_{k \textrm{ times}}),\, (0,1,\underbrace{0,\ldots, 0}_{k-1\textrm{ times}}),\,(0,0,\underbrace{1,\ldots, 1}_{k-1\textrm{ times}})$$ would do. Let $\sigma \in (\sym{\Delta})^{k+1}$ fix our set. Since $(0,\ldots,0)$ is the only repeated element in our set, it must be fixed by $\sigma$, meaning $\sigma$ fixes $0$ in each coordinate. But this immediately implies that $\sigma$ is trivial.

Similarly, when $a=2^k-2$, we choose two copies of $(0,\ldots, 0)$ and omit any five elements of $\Delta^{k+1}$ that leave us with regular partitions. For example, 
$$
(1,\underbrace{0, \ldots, 0}_{k\textrm{ times}}),\, 
(0,1, \underbrace{0, \ldots, 0}_{k-1\textrm{ times}}),\,
(0,0,1,\underbrace{0,\ldots, 0}_{k-2\textrm{ times}}),\,
(0,0,0,\underbrace{1,\ldots, 1}_{k-2 \textrm{ times}}),\,
(\underbrace{1,\ldots, 1}_{k+1\textrm{ times}})$$ will do. Let $\sigma \in (\sym{\Delta})^{k+1}$ fix our set. Since $(0,\ldots,0)$ is the only repeated element in our set, it must be fixed by $\sigma$, meaning $\sigma$ fixes $0$ in each coordinate. But this immediately implies that $\sigma$ is trivial. We do require $k \ge 3$ in order for the elements we removed to exist and be distinct.
\end{proof}

\begin{proof}[Proof of 
Theorem~$\ref{thrm:main2}$ for $b=2$.]This follows immediately from Corollary~\ref{cor:b=2alt} and Lemma~\ref{lem:new}, when $a\ge 5$. The result for $a\le 4$ follows with a computation.
\end{proof}

\end{document}